\documentclass[10pt]{amsart}
\usepackage{amssymb}
\usepackage[english]{babel}
\usepackage{fancyhdr}
\usepackage{colortbl}
\usepackage{enumerate}
\usepackage{graphicx} 
\usepackage{float}
\usepackage{mathrsfs}
\usepackage{cite}
\usepackage{kpfonts}
\usepackage{tcolorbox} 
\usepackage{enumitem}
\usepackage{bbm}
\usepackage{amsmath}
\usepackage{amsfonts}
\usepackage{amsthm} 
\usepackage{latexsym}
\usepackage{mathtools}           
\usepackage[linesnumbered, ruled, vlined]{algorithm2e}
\usepackage[left=2.3cm,right=2.3cm,top=2.65cm,bottom=2.65cm]{geometry}
\usepackage{tikz}
\usetikzlibrary{arrows,automata,shapes,calc,patterns}

\allowdisplaybreaks

\usepackage[colorlinks=true,urlcolor=green!50!black,
citecolor=red!60!black,linkcolor=blue!75!black,linktocpage,pdfpagelabels,
bookmarksnumbered,bookmarksopen]{hyperref}

\usepackage[capitalize,nameinlink,noabbrev]{cleveref}
\newtheorem{theorem}[]{Theorem}[section]

\numberwithin{equation}{section}

\newtheorem{definition}[theorem]{Definition}

\newtheorem{lemma}[theorem]{Lemma}

\newtheorem{remark}[theorem]{Remark}

\newtheorem{proposition}[theorem]{Proposition}

\usepackage{mathtools}
\DeclarePairedDelimiter\abs{\lvert}{\rvert}
\DeclarePairedDelimiter\norm{\lVert}{\rVert}
\DeclarePairedDelimiterX{\inner}[2]{\langle}{\rangle}{#1, #2}

\newcommand{\R}{\mathbb{R}}
\newcommand{\C}{\mathbb{C}}
\newcommand{\N}{\mathbb{N}}
\newcommand{\Rn}{\mathbb{R}^N}

\newcommand{\Lp}[1]{\mathit{L}^{#1}(\Rn)}

\newcommand{\normLp}[2]{{||{#1}||}_{#2}}
\newcommand{\dist}[2]{\text{dist}_{\mathit{H}}({#1},#2)}
\newcommand{\Hor}{{\mathit H}^1}
\newcommand{\Hoc}{{\mathit H}}

\renewcommand{\(}{\left(}
\renewcommand{\)}{\right)}

\renewcommand{\[}{\left[}
\renewcommand{\]}{\right]}

\newcommand{\weakto}{\rightharpoonup}
\newcommand{\diff}{\nabla}
\newcommand{\laplace}{\Delta}

\usepackage{etoolbox}

\graphicspath{{Figures/}}
\usepackage{float}
\usepackage[font=small,skip=0pt]{caption}

\begin{document}

\title[Orbital stability of ground states for a Sobolev critical  Schr\"odinger equation]{Orbital stability of ground states \\ for a Sobolev critical  Schr\"odinger equation}

\author[L. Jeanjean, J. Jendrej, T.T. Le and N. Visciglia]{Louis Jeanjean, Jacek Jendrej, Thanh Trung Le and Nicola Visciglia}

\address{
	\vspace{-0.25cm}
	\newline
	\textbf{{\small Louis Jeanjean}} 
	\newline \indent Laboratoire de Math\'{e}matiques (CNRS UMR 6623), Universit\'{e} de Bourgogne Franche-Comt\'{e}, Besan\c{c}on 25030, France}
\email{louis.jeanjean@univ-fcomte.fr} 

\address{
	\vspace{-0.35cm}
	\newline
	\textbf{{\small Jacek Jendrej}}
	\newline \indent CNRS and LAGA (CNRS UMR 7539), Universit\' e Sorbonne Paris Nord, Villetaneuse 93430, France}
\email{jendrej@math.univ-paris13.fr}

\address{
	\vspace{-0.25cm}
	\newline
	\textbf{{\small Thanh Trung Le }} 
	\newline \indent Laboratoire de Math\'{e}matiques (CNRS UMR 6623), Universit\'{e} de Bourgogne Franche-Comt\'{e}, Besan\c{c}on 25030, France}
\email{thanh\_trung.le@univ-fcomte.fr}

\address{
	\vspace{-0.35cm}
	\newline
	\textbf{{\small Nicola Visciglia }}
	\newline \indent Dipartimento di Matematica, Universit\`a Degli Studi di Pisa, Largo Bruno Pontecorvo, 5, 56127, Pisa, Italy}
\email{nicola.visciglia@unipi.it}

\thanks{J. Jendrej is supported by ANR-18-CE40-0028 project ESSED and Chilean projects FONDECYT 1170164 and France-Chile ECOS-Sud C18E06 project. N.V. is supported  by PRIN grant 2020XB3EFL and  by the Gruppo Nazionale per l’ Analisi Matematica, la Probabilità e le loro Applicazioni (GNAMPA) of the Istituzione Nazionale di Alta Matematica (INDAM)}

\date{}
\subjclass[2010]{}
\maketitle

\centerline {\em Dedicated to the memory of Professor Jean Ginibre}
\begin{abstract} 
	We study the existence of ground state standing waves, of prescribed mass, for the nonlinear Schr\"{o}dinger equation with mixed power nonlinearities
	\begin{align*}
	i \partial_t v + \Delta v + \mu v |v|^{q-2} + v |v|^{2^* - 2} = 0, \quad (t, x) \in \R \times \Rn, 
	\end{align*}
	where $N \geq 3$, $v: \R \times \Rn \to \C$, $\mu > 0$, $2 < q < 2 + 4/N $ and $2^* = 2N/(N-2)$ is the critical Sobolev exponent. We show that all ground states correspond to local minima of the associated Energy functional. Next, despite the fact that the nonlinearity is Sobolev critical, 
we show that the set of ground states is orbitally stable.
Our results settle a question raised by N. Soave \cite{Soave2020Sobolevcriticalcase}.
 \\
\end{abstract}

\section{Introduction}

In this paper, we study the existence and orbital stability of ground state standing waves of prescribed mass for the nonlinear Schr\"{o}dinger equation with mixed power nonlinearities
\begin{align}
i \partial_t v + \Delta v + \mu v |v|^{q-2} + v |v|^{2^* - 2} = 0, \quad (t, x) \in \R \times \Rn,  \label{NLS0}
\end{align}
where $N \geq 3$, $v: \R \times \Rn \to \C$, $\mu > 0$, $2 < q < 2 + \dfrac{4}{N}$ and $2^* = \dfrac{2N}{N-2}$. \\
The nonlinear Schr\"{o}dinger equation (NLS) with pure and mixed power nonlinearities has attracted much attention in the last decades. The local existence result for the pure power energy critical NLS
has been established in \cite{CazenaveWeissler1990}. The corresponding global existence and scattering
for defocusing quintic NLS in dimension $N=3$ 
has been established
in the papers \cite{Bourgain1999, CollianderKeelStaffilaniTakaokaTao2008} respectively 
in the radial and non-radial case.
We also quote the concentration-compactness/rigidity approach introduced
in \cite{KenigMerle2006} in order to study global existence and scattering in the focusing energy critical NLS
below the ground state.
Concerning the case of NLS with mixed nonlinearities let us quote 
\cite{TaoVisanZhang07, AkahoriIbrahimKikuchiNawa2012,AkahoriIbrahimKikuchiNawa2013,ChenMiaoZhao2016, ColesGustafson20,LewinRotaNodari2020,MiaoXuZhao2013,MiaoZhaoZheng2017}.

We recall that standing waves to \eqref{NLS0} are solutions of the form $v(t,x) = e^{-i\lambda t}u(x), \lambda \in \R$. Then the function $u(x)$ satisfies the equation
\begin{align}
-\laplace u - \lambda u - \mu \abs{u}^{q-2} u - \abs{u}^{2^*-2} u = 0 \quad \mbox{in } \Rn. \label{eqn:Laplace}
\end{align}

When looking for solutions to \eqref{eqn:Laplace} a possible choice is to consider $\lambda \in \R$ fixed and to search for solutions as critical points of the action functional
$$\mathcal{A}_{\lambda, \mu}(u) :=  \dfrac{1}{2} \normLp{\diff u}{2}^2 - \dfrac{\lambda}{2} \normLp{u}{2}^2    -
\dfrac{\mu}{q} \normLp{u}{q}^q - \dfrac{1}{2^*}\normLp{u}{2^*}^{2^*}.$$
In this case one usually focuses on the existence of minimal action solutions, namely of solutions minimizing $\mathcal{A}_{\lambda, \mu}$ among all non-trivial solutions. In that direction, we refer to \cite{AlvesSoutoMontenegro2012} where, relying on the pioneering work of Brezis-Nirenberg \cite{BrezisNirenberg1983}, the existence of positive real solutions for equations of the type of \eqref{eqn:Laplace} is  addressed in a very general setting; to \cite{AkahoriIbrahimKikuchiNawa2012,AkahoriIbrahimKikuchiNawa2013} which concerns the case where $q > 2 + 4/N$ and $\mu >0$; to \cite{ChenMiaoZhao2016,MiaoXuZhao2013} where the fixed $\lambda \in \R$ problem is analyzed for $q = 2 + 4/N$ and $\mu <0$; 
see also \cite{LewinRotaNodari2020} and the reference therein. \medskip

Alternatively, one can search for solutions to \eqref{eqn:Laplace} having a prescribed $\mathit{L}^2$-norm. Defining on $H:= H^1(\R^N, \C)$ the Energy functional
\begin{equation*}
F_{\mu}(u) := \dfrac{1}{2} \normLp{\diff u}{2}^2 - \dfrac{\mu}{q} \normLp{u}{q}^q - \dfrac{1}{2^*}\normLp{u}{2^*}^{2^*}
\end{equation*}
it is standard to check that $F_{\mu}$ is of class $C^1$ and that a critical point of $F_{\mu}$ restricted to the (mass) constraint
\begin{equation*}
S(c) := \{u \in H: \normLp{u}{2}^2 = c\}
\end{equation*}
gives rise to a solution to \eqref{eqn:Laplace}, satisfying $\normLp{u}{2}^2 = c.$

In this approach the parameter $\lambda \in \R$ arises as a Lagrange multiplier. In particular, $\lambda \in \R$ does depend on the solution and is not a priori given.  This approach, that we shall follow here, is relevant from the physical point of view, in particular, since the $L^2$ norm is a preserved quantity of the evolution and  since the variational characterization of such solutions is often a strong help to analyze their orbital stability, see for example, \cite{BellazziniJeanjeanLuo2013,CazenaveLions1982,Soave2020,Soave2020Sobolevcriticalcase}. \medskip

We shall focus on the existence of ground state solutions. 

\begin{definition}
	We say that  $u_c \in S(c)$ is a ground state solution to \eqref{eqn:Laplace} if it is a solution having minimal Energy among all the solutions which belong to $S(c)$. Namely, if
	$$\quad F_{\mu}(u_c) =  \displaystyle \inf \big\{F_{\mu}(u), u \in S(c), \big(F_\mu\big|_{S(c)}\big)'(u) = 0 \big\}.$$
\end{definition}
Note that this definition keeps a meaning even in situations where the Energy $F_{\mu}$ is unbounded from below on $S(c)$. Implicit in \cite{JEANJEAN1997}, this definition was formally introduced, on a related model, in 
\cite{BellazziniJeanjean2016} and is now becoming standard. \medskip

It is well-known that the study of problems with mixed nonlinearities and the type of results one can expect, depend on the behavior of the nonlinearities  at infinity, namely on the value of the various power exponents. In particular, this behavior determines whether the functionnal  is bounded from below on $S(c)$. One speaks of a mass subcritical case if it is bounded from below on $S(c)$ for any $c>0$, and of a mass supercritical case if the functional is unbounded from below on $S(c)$  for any $c>0$.  One also refers to  a mass critical case when the boundedness from below does depend on the value $c>0$. To be more precise, consider an equation of the form
\begin{align}
i \partial_t v + \Delta v + \mu v |v|^{p_1-2} + v |v|^{p_2 - 2} = 0, \quad (t, x) \in \R \times \Rn,  \label{NLS0E}
\end{align}
where it is assumed that $2< p_1 \leq p_2 \leq 2^*.$ The threshold exponent is the so-called $L^2$-critical exponent
$$p_c = 2 + \frac{4}{N}.$$
A very complete analysis of the various cases that may happen for \eqref{NLS0E}, depending on the values of $(p_1,p_2)$, has been provided recently in \cite{Soave2020,Soave2020Sobolevcriticalcase}. Let us just recall here some rough elements. If both $p_1$ and $p_2$ are strictly less than $p_c$ then the associated Energy functional is bounded from below on $S(c)$ and to find a ground state one looks for a global minimum on $S(c)$. The problem then directly falls into the setting covered by the Compactness by Concentration Principle introduced by P.L. Lions \cite{LIONS1984-1, LIONS1984-2} which, for more complicated equations, in particular non autonomous ones, is still a very active field. Such solutions are expected to be orbitally stable, see \cref{def:stability} below.  If $p_c \leq p_1 \leq p_2 \leq 2^*$, then the  Energy functional is unbounded from below on $S(c)$ but it is possible to show that a ground state exists. This ground state is characterized as a critical point of {\it mountain-pass type} and it lies at a strictly positive level of the Energy functional. Such ground states are expected to be orbitally unstable. We refer, for the link between the variational characterization of a solution and its instability, to the classical paper \cite{BerestyckiCazenave1981}, and to \cite{JEANJEAN1997,Lecoz2008,Soave2020,Soave2020Sobolevcriticalcase} for more recent developments. \medskip

In the case we consider here : $2 < p_1 < p_c <p_2 = 2^*$, the Energy functional  is thus unbounded from below on $S(c)$ but, as we shall see, the presence of the lower order, mass subcritical term $- \mu \normLp{u}{q}^q$ creates, for sufficiently small values of $c>0$, a geometry of local minima on $S(c).$ The presence of such geometry, in problems which are mass supercritical, had already been observed in several related situations. In \cite{BellazziniJeanjean2016, BellazziniBoussaidJeanjeanVisciglia17} for related scalar problems, in 
\cite{GouJeanjean2018} in the case of a system or \cite{NorisTavaresVerzini2019} for an evolution problem set on a bounded domain.  Actually, it was already observed on \eqref{NLS0} in \cite{Soave2020Sobolevcriticalcase}.\medskip

Precisely, for any fixed $\mu >0$, we shall find an explicit value $c_0 = c_0(\mu) >0$ such that, for any $c \in (0, c_0)$, there exists a set $V(c) \subset S(c)$ having the property that
\begin{equation}\label{well}
m(c) := \inf_{u \in V(c)} F_{\mu}(u) < 0 <  \inf_{u \in \partial V(c)}F_{\mu}(u).
\end{equation}
The sets $V(c)$ and $\partial V(c)$ are given by
$$V(c) := \{ u \in S(c) : \normLp{\diff u}{2}^2  < \rho_0\}, \qquad \partial V(c) := \{ u \in S(c) : \normLp{\diff u}{2}^2  = \rho_0\}$$
for a suitable $\rho_0 >0$, depending only on $c_0 >0$ but not on $c \in (0,c_0)$.
We also introduce the set
\begin{align*}
\mathcal{M}_c := \{u \in V(c) : F_{\mu}(u) = m(c)\}.
\end{align*}
Our first result is,
\begin{theorem}\label{thm-1}
	Let $N \geq 3$, $2 < q < 2 + \frac{4}{N}$. For any $\mu >0$ there exists a $c_0 = c_0(\mu) >0$ such that, for any $c \in (0, c_0)$,  $F_{\mu}$ restricted to $S(c)$
	has a ground state. This ground state is a (local) minimizer of $F_{\mu}$ in the set $V(c)$ and any ground state for $F_{\mu}$ on $S(c)$ is a local minimizer of $F_{\mu}$ on $V(c)$. In addition, if $(u_n) \subset V(c)$ is such that $F_{\mu}(u_n) \to m(c)$ then, up to translation,  $u_n \to u \in \mathcal{M}_c$ in $\mathit{H}^1(\Rn,\C)$. 
\end{theorem}
\begin{remark}
	The value of $c_0 = c_0(\mu) >0$ is explicit and is given in \eqref{eqn:5.14}-\eqref{eqn:5.14B}. In particular $c_0 >0$ can be taken arbitrary large by taking $\mu >0$ small enough.
\end{remark}

\begin{remark} \label{remark-th1}$ $
\begin{itemize}
\item[(i)] If $u \in S(c)$ is a ground state then the associated Lagrange multiplier $\lambda \in \R$ in \eqref{eqn:Laplace} satisfies $\lambda < 0$. This follows directly combining that $u$ being a solution to \eqref{eqn:Laplace} it satisfies $\normLp{\diff u}{2}^2 - \lambda \normLp{u}{2}^2    -
\mu \normLp{u}{q}^q - \normLp{u}{2^*}^{2^*} =0$ with the fact that $F_{\mu}(u) = m(c) <0.$
\smallbreak
\item[(ii)] There exists a ground state which is a real valued, positive, radially symmetric decreasing function. Indeed if $u \in S(c)$ is a ground state then its Schwartz symmetrization is clearly also a ground state.
\smallbreak
\item[(iii)] More globally, under the assumption of \cref{thm-1} it can be proved that, for any $c \in (0, c_0)$, $\mathcal{M}_c$ has the following structure:
	$$\mathcal{M}_c = \{ e^{i \theta}u, \mbox{ for some } \theta \in \R, u \in \tilde{\mathcal{M}}_c, u >0 \},$$
	where
	$$\tilde{\mathcal{M}}_c = \{ u \in S(c) \cap H^1(\R^N, \R), F_{\mu}(u) = m(c)\}.$$
	Indeed, this description directly follows from the convergence, up to translation, of the minimizing sequences of $F_{\mu}$ restricted 
	to $V(c)$, applying the argument of \cite[Section 3]{HajaiejStuart2004}. We leave the details to the interested reader.
\end{itemize}
\end{remark}

We shall now focus on the (orbital) stability of the set $\mathcal{M}_c$. 
Following the terminology of \cite{CazenaveLions1982}, see also \cite{HajaiejStuart2004}, we give the following definition.
\begin{definition}\label{def:stability}
	$Z \subset \Hoc$ is stable if : $Z \neq \emptyset $ and for any $v \in Z $ and any $ \varepsilon >0$, there exists a $\delta >0$ such that if $\varphi \in \Hoc$ satisfies $||\varphi- v||_{\Hoc} < \delta$ then 
	 $u_{\varphi}(t)$ is globally defined and $\inf_{z \in Z} ||u_{\varphi}(t) - z||_{\Hoc} < \varepsilon$ for all 
	$t \in \R$, where $u_{\varphi}(t)$ is the solution to \eqref{NLS0} corresponding to the initial condition $\varphi$.
\end{definition}

Notice that the orbital stability of the set $Z$ implies the global existence 
of solutions to \eqref{NLS0} for initial datum $\varphi$ close enough to the set $Z$. We underline that this fact 
is non trivial due to the critical exponent that appears in \eqref{NLS0}, even if the $H$ norm of the solution
is uniformly bounded on the lifespan of the solution.

The fact that ground states are characterized as local minima suggests, despite the problem being mass supercritical, that the set $\mathcal{M}_c$ could be orbitally stable. Actually, such orbital stability results have now been proved, on related problems (but always Sobolev subcritical) in several recent papers \cite{BellazziniBoussaidJeanjeanVisciglia17,GouJeanjean2018,Soave2020}.  Along this line we now present the main result of this paper.
\begin{theorem}\label{thm-2}
	Let $N \geq 3$, $2 < q < 2 + \frac{4}{N}$,  $\mu >0$ and  $c_0 = c_0(\mu) >0$ be given in \cref{thm-1}. Then, for any $c \in (0, c_0),$ the set $\mathcal{M}_c$ is compact, up to translation, and it is orbitally stable.
\end{theorem}

In \cite{Soave2020Sobolevcriticalcase}, Soave studied  equation \eqref{NLS0} and derived, for any small $c>0$ depending on $\mu>0$, an existence result which is very similar to the one contained in \cref{thm-1}, see \cite[Theorem 1.1]{Soave2020Sobolevcriticalcase}. Actually, the motivation of our study originated from \cite{Soave2020Sobolevcriticalcase} where what is now our \cref{thm-2} was proposed as an open problem.
However, it does not seem possible to use \cite[Theorem 1.1]{Soave2020Sobolevcriticalcase} as a starting point to prove \cref{thm-2}. The existence of a ground state in \cite[Theorem 1.1]{Soave2020Sobolevcriticalcase} is obtained through the study of one particular (locally) minimizing sequence which is radially symmetric. As already explained in \cite{Soave2020Sobolevcriticalcase}, to obtain the orbital stability of the set $\mathcal{M}_c$, following the classical approach laid down in \cite{CazenaveLions1982}, two ingredients are essential. First, the relative compactness, up to translation, of all minimizing sequences for $F_{\mu}$ on $V(c)$, as guaranteed by our \cref{thm-1}. Secondly, the global existence of solutions to \eqref{NLS0} for initial data close to $\mathcal{M}_c$.

To obtain the relative compactness of all minimizing sequences, the fact that one minimizes only on a subset of $S(c)$, in contrast to a global minimization on all $S(c)$, increases the difficulty to rule out a possible {\it dichotomy}. Different strategies have been recently implemented to deal with this issue \cite{BellazziniBoussaidJeanjeanVisciglia17,GouJeanjean2018,Soave2020}, all relying on a suitable choice of the set where the local minima are searched. In the presence of a Sobolev critical term an additional difficulty arises. In a Sobolev subcritical setting, if a sequence $(v_n) \subset S(c)$ is {\it vanishing} then applying \cite[Lemma I.1]{LIONS1984-2} one would immediately get 
$$\liminf_{n \to \infty} F_{\mu}(v_n) = \liminf_{n \to \infty}\frac{1}{2}||v_n||_2^2 \geq 0.$$
Thus the {\it vanishing} can directly be ruled out knowing that $m(c) <0$. Here \cite[Lemma I.1]{LIONS1984-2} does not apply anymore; the term $||v_n||_{2^*}$ may not go to $0$ if $(v_n)$ is {\it vanishing}. Thus we need a better understanding of this possible loss of compactness and this leads to our definition of the set $V(c)$.

As to the global existence of solutions to \eqref{NLS0}, it is also affected by the presence of the Sobolev critical exponent. In Sobolev subcritical cases, it is well known \cite{Cazenave2003semilinear} that if, for an initial datum $\varphi \in H$, the maximum time of existence $T_{\varphi}^{max}>0$ is finite then necessarily the corresponding solution $v$ satisfies $||\nabla v(t)||_2 \to + \infty$ as $t \to T_{\varphi}^{max}$. Thus, a uniform a priori bound on $||\nabla v(t)||_2$ yields global existence. Note that, by conservation of the Mass and Energy, in view of \eqref{well},
for an initial datum in $V(c) \cap \{u \in S(c) : F_{\mu}(u) <0\}$, the evolution takes place in the (bounded) set $V(c)$. Thus, in a subcritical setting, the global existence would follow directly. However, in our case it is unknown if the previous blow-up alternative holds and hence, we cannot deduce global existence just since the evolution takes place in $V(c)$, see \cite[Theorem 4.5.1]{Cazenave2003semilinear} or \cite[Proposition 3.2]{TaoVisanZhang07} for more details.  To overcome this difficulty, building on the pioneering work of Cazenave-Weissler \cite{CazenaveWeissler1990}, see also \cite[Section 4.5]{Cazenave2003semilinear} , we first derive an upper bound on the propagator $e^{it \Delta}$ which provides a kind of uniform local existence result, see \cref{prop:cauchy}. Next, using the information that all minimizing sequences are, up to translation, compact and also specifically and crucially that $\mathcal{M}_c$ is compact, up to translation, we manage to show that, for initial data sufficiently close to the set $\mathcal{M}_c$ the global existence holds and this leads to the orbital stability of $\mathcal{M}_c$, proving \cref{thm-2}. \medskip 

We point out that, in order to prove \cref{thm-2}, we have only established the global existence of solutions for initial data {\it close} to $\mathcal{M}_c$. We believe it would be interesting to inquire if the global existence holds {\it away} from $\mathcal{M}_c$, typically for any initial data in $V(c) \cap \{u \in S(c) : F_{\mu}(u) <0\}$. If so, investigating the long time behavior of these solutions would be worth to. Our guess is that these solutions evolve toward the sum of an element of $\mathcal{M}_c$ and a part which scatter. However, so far nothing is known in that direction. \medskip

The paper is organized as follows. \cref{Section-3}, is devoted to clarifying the local minima structure and to establish the convergence, up to translation, of all minimizing sequences for $F_{\mu}$ on $V(c)$. The proof of \cref{thm-1} is then given. In \cref{Section-4}, we establish \cref{prop:cauchy}. Finally, in \cref{Section-5}, we prove \cref{thm-2} which states the orbital stability of the set $\mathcal{M}_c$. \medskip

{\bf Notation :} We write $\Hoc$ for $\mathit{H}^1(\Rn,\C)$.
For $p \geq 1$, the $\mathit{L}^p$-norm of $u \in \Hoc$ (or of $u \in \Hor$) is denoted by $\normLp{u}{p}$.


\section{The variational problem} \label{Section-3}

We shall make use of the following classical inequalities :
For any $N \geq 3$  there exists an optimal constant
$\mathcal{S} > 0$ depending only on $N$, such that
\begin{align}\label{Sobolev-I}
\mathcal{S} \norm{f}_{2^*}^{2} \leq \norm{\diff f}_{2}^{2} , \qquad \forall f \in H, \quad \mbox{(Sobolev inequality)}
\end{align}
see \cite[Theorem IX.9]{Brezis1983}. 
If
$N \geq 2$ and $p \in [2, \frac{2N}{N-2})$ then
\begin{align}\label{Gagliardo-Nirenberg-I}
\norm{f}_{p} \leq C_{N,p} \norm{\diff f}_{2}^{\beta} \norm{f}_{2}^{(1-\beta)}, \qquad \mbox{with } \beta = N\(\dfrac{1}{2} - \dfrac{1}{p}\) \quad \mbox{(Gagliardo-Nirenberg inequality),}
\end{align}
for all $f \in H$, see \cite{Nirenberg1959}.
Now, letting
\begin{align*}
\alpha_0 := \dfrac{N(q-2)}{2} -2, \qquad  \alpha_1:= \dfrac{2N - q(N-2)}{2}, \qquad \alpha_2: = \dfrac{4}{N-2},
\end{align*}
we consider the function $f(c,\rho)$ defined on $(0, \infty) \times (0, \infty)$ by
\begin{align*}
f(c,\rho) = \frac{1}{2} - \frac{\mu}{q} C_{N,q}^q \rho^{\frac{\alpha_0}{2} } c^{\frac{\alpha_1}{2} } - \frac{1}{2^*} \dfrac{1}{\mathcal{S}^{\frac{2^*}{2}}} \rho^{\frac{\alpha_2}{2} },
\end{align*}
and, for each $c \in (0, \infty)$, its restriction $g_c(\rho)$ defined on $(0, \infty)$ by $\rho \mapsto g_c(\rho) := f(c, \rho).$
\\
For future reference, note that for any $N \geq 3 $,  $\alpha_0 \in (-2,0)$, 
$\alpha_1 \in \[\dfrac{4}{N}, 2\)$ and $\alpha_2 \in (0,4].$

\begin{lemma}\label{lemma:5.1}
	For each $c > 0$, the function $g_c(\rho)$ has a unique global maximum and the maximum value satisfies
	\begin{align*}  
	\begin{cases} 
	\displaystyle \max_{\rho > 0} g_{c}(\rho) > 0 \quad \mbox{if} \quad c < c_0,\\
	\displaystyle \max_{\rho > 0} g_{c}(\rho) = 0 \quad \mbox{if} \quad c = c_0,\\
	\displaystyle \max_{\rho > 0} g_{c}(\rho) < 0 \quad \mbox{if} \quad c > c_0,
	\end{cases}
	\end{align*}
	where
	\begin{align}\label{eqn:5.14}
	c_0 := \(\dfrac{1}{2K}\)^{\frac{N}{2}} > 0, 
	\end{align}
	with
	\begin{align}\label{eqn:5.14B}
	K := \frac{\mu}{q} C_{N,q}^q \[- \dfrac{\alpha_0}{\alpha_2} \frac{\mu C_{N,q}^q 2^* \mathcal{S}^{\frac{2^*}{2}}}{q} \]^{\frac{\alpha_0}{\alpha_2 - \alpha_0}} + \frac{1}{2^*} \dfrac{1}{\mathcal{S}^{\frac{2^*}{2}}} \[- \dfrac{\alpha_0}{\alpha_2} \frac{\mu C_{N,q}^q 2^* \mathcal{S}^{\frac{2^*}{2}}}{q} \]^{\frac{\alpha_2}{\alpha_2 - \alpha_0}} > 0.
	\end{align}
\end{lemma}
\begin{proof}
	By definition of $g_c(\rho)$, we have that
	\begin{align*}
	g_c'(\rho) = - \dfrac{\alpha_0}{2} \frac{\mu}{q} C_{N,q}^q \rho^{\frac{\alpha_0}{2} -1} c^{\frac{\alpha_1}{2} }
	- \dfrac{\alpha_2}{2} \frac{1}{2^*} \dfrac{1}{\mathcal{S}^{\frac{2^*}{2}}} \rho^{\frac{\alpha_2}{2} -1 }.
	\end{align*}
	Hence, the equation $g_c'(\rho) = 0$ has a unique solution given by
	\begin{align}\label{maxL}
	\rho_c = \[- \dfrac{\alpha_0}{\alpha_2} \frac{\mu C_{N,q}^q 2^* \mathcal{S}^{\frac{2^*}{2}}}{q} \]^{\frac{2}{\alpha_2 - \alpha_0}}  c^{\frac{\alpha_1}{\alpha_2 - \alpha_0}}.
	\end{align}
	Taking into account that $g_c(\rho) \to -\infty$ as $\rho \to 0$ and $g_c(\rho) \to -\infty$ as $\rho \to \infty$, we obtain that $\rho_c$ is the unique global maximum point of $g_c(\rho)$ and the maximum value is
	\begin{align*}
	\max_{\rho > 0} g_{c}(\rho) &= \frac{1}{2} - \frac{\mu}{q} C_{N,q}^q \[- \dfrac{\alpha_0}{\alpha_2} \frac{\mu C_{N,q}^q 2^* \mathcal{S}^{\frac{2^*}{2}}}{q} \]^{\frac{\alpha_0}{\alpha_2 - \alpha_0}}  c^{\frac{\alpha_0\alpha_1}{2(\alpha_2 - \alpha_0)}} c^{\frac{\alpha_1}{2} } 
	 - \frac{1}{2^*} \dfrac{1}{\mathcal{S}^{\frac{2^*}{2}}} \[- \dfrac{\alpha_0}{\alpha_2} \frac{\mu C_{N,q}^q 2^* \mathcal{S}^{\frac{2^*}{2}}}{q} \]^{\frac{\alpha_2}{\alpha_2 - \alpha_0}}  c^{\frac{\alpha_1\alpha_2}{2(\alpha_2 - \alpha_0)}} \\
	&= \frac{1}{2} - \frac{\mu}{q} C_{N,q}^q \[- \dfrac{\alpha_0}{\alpha_2} \frac{\mu C_{N,q}^q 2^* \mathcal{S}^{\frac{2^*}{2}}}{q} \]^{\frac{\alpha_0}{\alpha_2 - \alpha_0}}  c^{\frac{\alpha_1\alpha_2}{2(\alpha_2 - \alpha_0)}} 
	 - \frac{1}{2^*} \dfrac{1}{\mathcal{S}^{\frac{2^*}{2}}} \[- \dfrac{\alpha_0}{\alpha_2} \frac{\mu C_{N,q}^q 2^* \mathcal{S}^{\frac{2^*}{2}}}{q} \]^{\frac{\alpha_2}{\alpha_2 - \alpha_0}}  c^{\frac{\alpha_1\alpha_2}{2(\alpha_2 - \alpha_0)}}\\
	&= \dfrac{1}{2} - K c^{\frac{2}{N}}.
	\end{align*}
	By the definition of $c_0$, we have that $\displaystyle \max_{\rho > 0} g_{c_0}(\rho) = 0$, and hence the lemma follows.
\end{proof}

\begin{lemma}\label{LL6-1}
	Let $(c_1, \rho_1) \in  (0, \infty) \times (0, \infty)$ be such that $f(c_1, \rho_1) \geq 0$. Then for any $c_2 \in (0,c_1]$, we have that
	\begin{align*}
	f(c_2, \rho_2) \geq 0 \quad \mbox{if} \quad \rho_2 \in \[  \displaystyle \frac{c_2}{c_1}\rho_1, \rho_1\].
	\end{align*}
\end{lemma}
\begin{proof}
	Since $c \to f(\cdot, \rho)$ is a non-increasing  function we clearly have that 
	\begin{align}
	f(c_2, \rho_1) \geq f(c_1, \rho_1) \geq 0. \label{eqn:5.20}
	\end{align}
	Now taking into account that $\alpha_0 + \alpha_1 = q-2 > 0 $ we have, by direct calculations, that 
	\begin{align}
	f\(c_2, \dfrac{c_2}{c_1}\rho_1\) \geq f(c_1, \rho_1) \geq 0. \label{eqn:5.21}
	\end{align}
	We observe that if $g_{c_2}(\rho') \geq 0$ and $g_{c_2}(\rho'') \geq 0$ then
	\begin{align}
	f(c_2, \rho) = g_{c_2}(\rho) \geq 0  \quad \mbox{for any} \quad \rho \in [\rho', \rho'']. \label{eqn:5.22}
	\end{align}
	Indeed, if $g_{c_2}(\rho) < 0$ for some $\rho \in (\rho', \rho'')$ then there exists a local minimum point on $ (\rho_1, \rho_2)$ and this contradicts the fact that the function $g_{c_2}(\rho)$  has a unique  critical point
	which has to coincide necessarily with its unique global maximum  (see \cref{lemma:5.1}).
	By \eqref{eqn:5.20}, \eqref{eqn:5.21}, we can choose $\rho' =  (c_2/c_1) \rho_1$ and $\rho'' = \rho_1$, and  \eqref{eqn:5.22} implies the lemma.
\end{proof}

\begin{lemma}\label{Lemma-L1} For any $u \in S(c)$,  we have that
	\begin{equation*}
	F_{\mu}(u) \geq \normLp{\diff u}{2}^2 f(c, \normLp{\diff u}{2}^2).
	\end{equation*}
\end{lemma}
\begin{proof}
	Applying the Gagliardo-Nirenberg inequality \eqref{Gagliardo-Nirenberg-I}
	and the Sobolev inequality \eqref{Sobolev-I} we obtain that, for any $u \in S(c)$,
	\begin{align*}
	F_{\mu}(u) &= \frac{1}{2} \normLp{\diff u}{2}^2 - \frac{\mu}{q} \normLp{u}{q}^q - \frac{1}{2^*}\normLp{u}{2^*}^{2^*} 
	\geq \frac{1}{2} \normLp{\diff u}{2}^2 - \frac{\mu}{q} C_{N,q}^q \norm{\diff u}_{2}^{\alpha_0+2} \norm{u}_{2}^{\alpha_1} - \frac{1}{2^*} \dfrac{1}{\mathcal{S}^{\frac{2^*}{2}}} \normLp{\diff u}{2}^{2^*} \\
	&= \normLp{\diff u}{2}^2 \[ \frac{1}{2} - \frac{\mu}{q} C_{N,q}^q \norm{\diff u}_{2}^{\alpha_0 } \norm{u}_{2}^{\alpha_1} - \frac{1}{2^*} \dfrac{1}{\mathcal{S}^{\frac{2^*}{2}}} \normLp{\diff u}{2}^{\alpha_2}\]
	= \normLp{\diff u}{2}^2 f( \norm{u}_{2}^2, \normLp{\diff u}{2}^2). 
	\end{align*}
	The lemma is proved.
\end{proof}

Now let $c_0>0$ be given by \eqref{eqn:5.14}  and $\rho_0 := \rho_{c_0} >0$ being determined by \eqref{maxL}. Note that by \cref{lemma:5.1} and \cref{LL6-1}, we have that $f(c_0, \rho_0) = 0$ and $f(c, \rho_0) > 0$ for all $c \in (0, c_0)$. 
We define
\begin{align*}
B_{\rho_0} := \{u \in \Hoc: \normLp{\diff u}{2}^2 < \rho_0\} \quad \mbox{and} \quad V(c) := S(c) \cap B_{\rho_0}.
\end{align*}
We shall now consider the following local minimization problem:  for any $c \in (0, c_0)$,
\begin{equation}\label{L6-3}
m(c) := \inf_{u \in V(c)} F_{\mu}(u).
\end{equation}

\begin{lemma}\label{Lemma-structure} For any $c \in (0, c_0)$, the following properties hold, 
	\begin{enumerate}[label=(\roman*), ref = \roman*]
		\item\label{point:5L.5i} 
		$$m(c) = \inf_{u \in V(c)} F_{\mu}(u) < 0 <  \inf_{u \in \partial V(c)}F_{\mu}(u).$$
		\item\label{point:5L.5ii} If $m(c)$ is reached, then any ground state is contained in $V(c)$.
	\end{enumerate}
\end{lemma}

\begin{proof}
	(\ref{point:5L.5i}) For any $u \in \partial V(c)$ we have $\normLp{\diff u}{2}^2 = \rho_0$. Thus, using \cref{Lemma-L1}, we get
	\begin{align*}
	F_{\mu}(u)  \geq \normLp{\diff u}{2}^2 f( \norm{u}_{2}^2, \normLp{\diff u}{2}^2)
	= \rho_0 f(c, \rho_0) >0.  
	\end{align*}
	Now let $u \in S(c)$ be arbitrary but fixed. For $s \in (0, \infty)$ we set
	\begin{align*}
	u_s(x) := s^{\frac{N}{2}} u(s x).
	\end{align*}
	Clearly $u_s \in S(c)$ for any $s \in (0, \infty)$. We define on $(0, \infty)$ the map,
	\begin{align*}
	\psi_u(s) := F_{\mu} (u_s) = \dfrac{s^2}{2} \normLp{\diff u}{2}^2 - \dfrac{\mu}{q} s^{\frac{N(q-2)}{2}} \normLp{u}{q}^q - \dfrac{s^{2^*}}{2^*} \normLp{u}{2^*}^{2^*}.
	\end{align*}
	Taking into account that
	\begin{align*}
	\dfrac{N(q-2)}{2} < 2 \qquad \mbox{and} \qquad 2^* > 2, 
	\end{align*}
	we see that $\psi_{u}(s) \to 0^-$, as $s \to 0$. Therefore, there exists $s_0 >0$ small enough such that $\normLp{\diff (u_{s_0})}{2}^2 = s_0^2 \normLp{\diff u}{2}^2 < \rho_0$ and $F_{\mu} (u_{s_0}) = \psi_{\mu}(s_0) < 0$. This implies that $m(c) < 0$.
	
	(\ref{point:5L.5ii}) It is well known, see for example \cite[Lemma 2.7]{JEANJEAN1997}, that all critical points of $F_{\mu}$ restricted to $S(c)$ belong to the Pohozaev's type set 
	$$\mathcal{P}_c := \{ u \in S(c) : \mathcal{P}(u)=0 \}$$
	where
	\begin{equation*}
	\mathcal{P}(u):= \normLp{\diff u}{2}^2 - \dfrac{\mu N(q-2)}{2q} \normLp{u}{q}^q - \normLp{u}{2^*}^{2^*}.
	\end{equation*}
	Also a direct calculation shows that, for any $v \in S(c)$ and any $s \in (0, \infty)$, 
	\begin{equation}\label{link}
	\psi_v'(s) = \frac{1}{s}\mathcal{P}(v_s).
	\end{equation}
	Here $\psi_v'$ denotes the derivative of $\psi_v$ with respect to $s \in (0, \infty)$. Finally, observe that any $u \in S(c)$ can be written as $u = v_s$ with $v \in S(c)$, $||\nabla v||_2 =1$ and $s \in (0, \infty).$
	
Since the set $\mathcal{P}_c$ contains all the ground states (if any), we deduce from \eqref{link} that if $w \in S(c)$ is a ground state there exists a $v \in S(c)$, $||\nabla v||_2^2=1$ and a $s_0 \in (0, \infty)$ such that $w= v_{s_0}$, $F_{\mu}(w) = \psi_v(s_0)$ and $\psi_{v}'(s_0)=0$. Namely, $s_0 \in (0, \infty)$ is a zero of the function $\psi_{v}'.$
	
Now, since $\psi_v(s) \to 0^-$, $||\nabla v_s||_2 \to 0,$ as $s \to 0$ and $\psi_v(s) = F_{\mu}(v_s) \geq 0$ when $v_s \in \partial V(c) = \{u \in S(c) : ||\nabla u||_2^2 = \rho_0 \}$, necessarily $\psi_v'$ has a first zero 
$s_1 >0$ corresponding to a local minima. In particular,  $v_{s_1} \in V(c)$ and $F(v_{s_1}) = \psi_v(s_1) <0.$ Also, from $ \psi_v(s_1) <0,$ $\psi_v(s) \geq  0$ when $v_s \in \partial V(c)$ and $\psi_v(s) \to - \infty$ as $s \to  \infty$,  $\psi_v$ has a second zero $s_2 >s_1$ corresponding to a local maxima of $\psi_v$. Since $v_{s_2}$ satisfies $F(v_{s_2})= \psi_v(s_2) \geq 0$, we have that $m(c) \leq F(v_{s_1}) < F(v_{s_2})$. In particular, since $m(c)$ is reached, $v_{s_2}$ cannot be a ground state.

To conclude the proof of (ii) it then just suffices to show that $\psi_v'$ has at most two zeros, since this will imply $s_0 = s_1$ and $w = v_{s_0} = v_{s_1} \in V(c)$. However, this is equivalent to showing that the function
	$$s \mapsto \frac{\psi'_u(s)}{s}$$
	has at most two zeros. We have
	$$\theta(s) := \frac{\psi'_u(s)}{s} = \normLp{\diff u}{2}^2 - \dfrac{\mu N(q-2)}{2q} s^{\alpha_0} \normLp{u}{q}^q - s^{\alpha_2} \normLp{u}{2^*}^{2^*}$$
	and
	\begin{align*}
	\theta'(s) = - \alpha_0 \dfrac{\mu N(q-2)}{2q} s^{\alpha_0 - 1} \normLp{u}{q}^q - \alpha_2 s^{\alpha_2-1} \normLp{u}{2^*}^{2^*}.
	\end{align*}
	Since $\alpha_0 <0$ and $\alpha_2 >0$, the equation $\theta'(s) =0$ has a unique solution, and $\theta(s)$ has indeed at most two zeros.
\end{proof}

We now introduce the set
\begin{align}\label{set-stable}
\mathcal{M}_c := \{u \in V(c) : F_{\mu}(u) = m(c)\}.
\end{align}
 The main aim of this section is the following result.
\begin{theorem} \label{theorem:LT-L}
	For any $c \in (0, c_0)$, if $(u_n) \subset B_{\rho_0}$ is such that $\normLp{u_n}{2}^2 \to c$ and $F_{\mu}(u_n) \to m(c)$ then, up to translation,  $u_n \overset{ H}\to u \in \mathcal{M}_c$. In particular the set $\mathcal{M}_c$ is compact in $H$, up to translation.
\end{theorem}
\cref{theorem:LT-L} will both imply the existence of a ground state but also, as it may be expected, will be a crucial step to derive the orbital stability of the set $\mathcal{M}_c$.\\

In order to prove \cref{theorem:LT-L} we collect some properties of $m(c)$ defined in \eqref{L6-3}.
\begin{lemma} \label{lemma:5.5}
	It holds that
	\begin{enumerate}[label=(\roman*), ref = \roman*]
		\item\label{point:5.5ii} $c \in (0,c_0) \mapsto m(c)$ is a continuous mapping.
		\item\label{point:5.5iii} Let $c \in (0,c_0)$. We have for all $\alpha \in (0,c)$ : $m(c) \leq m(\alpha) + m(c-\alpha)$ and if $m(\alpha)$ or $m(c-\alpha)$ is reached  then  the inequality is strict.
	\end{enumerate}
\end{lemma}
\begin{proof}
	(\ref{point:5.5ii}) Let $c \in (0, c_0)$ be arbitrary and $(c_n) \subset (0, c_0)$ be such that $c_n \to c$. From the definition of $m(c_n)$ and since $m(c_n) <0$, see \cref{Lemma-structure} (\ref{point:5L.5i}),
	for any $\varepsilon >0$ sufficiently small, there exists $u_n \in V(c_n)$ such that
	\begin{align}
	F_{\mu}(u_n) \leq m(c_n) + \varepsilon \quad \mbox{and} \quad F_{\mu}(u_n) <0. \label{eqn:5.9}
	\end{align}
	We set $\displaystyle y_n := \sqrt{\frac{c}{c_n}} u_n$ and hence $y_n \in S(c)$. We have that $y_n \in V(c)$. Indeed, if $c_n \geq c$, then
	\begin{align*}
	\normLp{\diff y_n}{2}^2 = \frac{c}{c_n} \normLp{\diff u_n}{2}^2 \leq \normLp{\diff u_n}{2}^2 < \rho_0.
	\end{align*}
	If $c_n < c$, by \cref{LL6-1}, we have $\displaystyle f(c_n, \rho) \geq 0$ for any $\rho \in \[ \dfrac{c_n}{c} \rho_0, \rho_0\]$. Hence, we deduce from \cref{Lemma-L1} and \eqref{eqn:5.9} that $f(c_n, \|\nabla u_n\|_2^2) < 0$, thus $\|\nabla u_n\|_2^2 < \frac{c_n}{c}\rho_0$ and
	\begin{align*}
	\normLp{\diff y_n}{2}^2 = \frac{c}{c_n} \normLp{\diff u_n}{2}^2 < \frac{c}{c_n} \dfrac{c_n}{c} \rho_0 = \rho_0.
	\end{align*}
	Since $y_n \in V(c)$ we can write
	\begin{align*}
	m(c) \leq F_{\mu}(y_n) = F_{\mu}(u_n) + [F_{\mu}(y_n) - F_{\mu}(u_n)] 
	\end{align*}
	where
	\begin{equation*}
	F_{\mu}(y_n) - F_{\mu}(u_n) = - \frac{1}{2}(\frac{c}{c_n}-1)\normLp{\diff u_n}{2}^2 - \frac{\mu}{q} \big[ (\frac{c}{c_n})^\frac q2 - 1 \big]  \normLp{u_n}{q}^q -  \frac{1}{2^*} [ (\frac{c}{c_n})^{ \frac{2^*}2} - 1 ] \normLp{u_n}{2^*}^{2^*}.
	\end{equation*}
	Since $\normLp{\diff u_n}{2}^2 < \rho_0$, also $\normLp{u_n}{q}^q$ and $\normLp{u_n}{2^*}^{2^*}$ are uniformly bounded. 
	Thus, as $n \to \infty$ we have
	\begin{align}
	m(c) \leq F_{\mu}(y_n) = F_{\mu}(u_n) + o_n(1).  \label{eqn:5.10}
	\end{align}
	Combining \eqref{eqn:5.9} and \eqref{eqn:5.10}, we get
	\begin{align*}
	m(c) \leq m(c_n) + \varepsilon + o_n(1).
	\end{align*}
	Now, let $u \in V(c)$ be such that
	\begin{equation*}
	F_{\mu}(u) \leq m(c) + \varepsilon \quad \mbox{and} \quad F_{\mu}(u) <0. 
	\end{equation*}
	Set $u_n := \sqrt{\frac{c_n}{c}} u$ and hence $u_n \in S(c_n)$.
	Clearly, $\|\nabla u\|_2^2 < \rho_0$ and $c_n \to c$ imply $\|\nabla u_n\|_2^2 < \rho_0$ for $n$ large enough,
	so that $u_n \in V(c_n)$. Also, $F_{\mu}(u_n) \to F_\mu(u)$. We thus have
	\begin{equation*}
	m(c_n) \leq F_\mu(u_n) = F_\mu(u) + [F_{\mu}(u_n) - F_{\mu}(u)] \leq m(c) + \varepsilon + o_n(1).
	\end{equation*}

	Therefore, since $\varepsilon > 0$ is arbitrary, we deduce that $m(c_n) \to m(c)$. The point \eqref{point:5.5ii} follows.
	
	(\ref{point:5.5iii}) Note that, fixed $\alpha \in (0,c)$, it is sufficient to prove that the following holds
	\begin{equation}\label{L6-4}
	\forall \theta \in \(1, \frac{c}{\alpha}\] :  m(\theta \alpha) \leq \theta m(\alpha)
	\end{equation}
	and that, if $m(\alpha)$ is reached, the inequality is strict. Indeed, if \eqref{L6-4} holds then we have
	\begin{align*}
	m(c)=\frac{c-\alpha}{c}m(c)+\frac{\alpha}{c}m(c)=\frac{c-\alpha}{c}m\left( \frac{c}{c-\alpha}(c-\alpha) \right)+\frac{\alpha}{c}m\left( \frac{c}{\alpha} \alpha \right) \leq m(c-\alpha)+m(\alpha),
	\end{align*}
	with a strict inequality if $m(\alpha)$ is reached. To prove that \eqref{L6-4} holds, note that in view of \cref{Lemma-structure} (\ref{point:5L.5i}), for any $\varepsilon >0$ sufficiently small, there exists a $u \in V(\alpha)$ such that
	\begin{align}
	F_{\mu}(u) \leq m(\alpha) + \varepsilon \quad \mbox{and} \quad F_{\mu}(u) <0. \label{eqn:5.11}
	\end{align}
	In view of \cref{LL6-1},  $\displaystyle f(\alpha, \rho) \geq 0$ for any $\rho \in \[ \dfrac{\alpha}{c} \rho_0, \rho_0\]$.
	Hence, we can deduce from \cref{Lemma-L1} and \eqref{eqn:5.11} that
	\begin{equation}\label{wellinside}
	||\nabla u||_2^2 < \frac{\alpha}{c} \rho_0.
	\end{equation} 
	Consider now $v = \sqrt{\theta} u$.  We first note that $||v||_2^2 = \theta ||u||_2^2 = \theta \alpha$ and also, because of \eqref{wellinside},
	$||\nabla v||_2^2 = \theta ||\nabla u||_2^2 < \rho_0.$
	Thus $v \in V(\theta \alpha)$ and we can write
	\begin{align*}
	m(\theta \alpha) &\leq F_{\mu}(v) = \dfrac{1}{2} \theta \normLp{\diff u}{2}^2 - \frac{\mu}{q} \theta^{\frac{q}{2}} \normLp{u}{q}^q - \dfrac{1}{2^*} \theta^{\frac{2^*}{2}} \normLp{u}{2^*}^{2^*}  
	< \dfrac{1}{2} \theta \normLp{\diff u}{2}^2 - \frac{\mu}{q} \theta \normLp{u}{q}^q - \dfrac{1}{2^*} \theta \normLp{u}{2^*}^{2^*}  \\
	&= \theta \(\dfrac{1}{2} \normLp{\diff u}{2}^2 - \frac{\mu}{q}  \normLp{u}{q}^q - \dfrac{1}{2^*} \normLp{u}{2^*}^{2^*}\) 
	= \theta F_{\mu}(u) \leq \theta (m(\alpha) + \varepsilon). 
	\end{align*}
	Since $\varepsilon > 0$ is arbitrary, we have that  $m(\theta \alpha) \leq \theta m(\alpha)$. If $m(\alpha)$ is reached then we can let $\varepsilon = 0$ in \eqref{eqn:5.11} and thus the strict inequality follows.
\end{proof}

\begin{lemma} \label{lemma:5.6}
	Let $(v_n) \subset B_{\rho_0}$ be such that $\normLp{v_n}{q} \to 0$. Then there exists a $\beta_0 > 0$ such that
	\begin{align*}
	F_{\mu}(v_n) \geq \beta_0 ||\nabla v_n||_2^2 + o_n(1).
	\end{align*}
\end{lemma}
\begin{proof}
	Indeed, using the Sobolev inequality \eqref{Sobolev-I}, we obtain that
	\begin{align*}
	F_{\mu}(v_n) &= \frac{1}{2} \normLp{\diff v_n}{2}^2  - \frac{1}{2^*}\normLp{v_n}{2^*}^{2^*} + o_n(1)  
	\geq \frac{1}{2} \normLp{\diff v_n}{2}^2  - \frac{1}{2^*} \dfrac{1}{\mathcal{S}^{\frac{2^*}{2}}} \normLp{\diff v_n}{2}^{2^*} + o_n(1) \\
	&= \normLp{\diff v_n}{2}^2 \[ \frac{1}{2}  - \frac{1}{2^*} \dfrac{1}{\mathcal{S}^{\frac{2^*}{2}}} \normLp{\diff v_n}{2}^{\alpha_2}\] + o_n(1) 
	 \geq  \normLp{\diff v_n}{2}^2 \[ \frac{1}{2}  - \frac{1}{2^*} \dfrac{1}{\mathcal{S}^{\frac{2^*}{2}}} \rho_0^{\frac{\alpha_2}{2}}\] + o_n(1).  
	\end{align*}
	Now, since $f(c_0,\rho_0) = 0,$ we have that
	\begin{align*}
	\beta_0 := \[ \frac{1}{2}  - \frac{1}{2^*} \dfrac{1}{\mathcal{S}^{\frac{2^*}{2}}} \rho_0^{\frac{\alpha_2}{2}}\] = \frac{\mu}{q} C_{N,q}^q \rho_0^{\frac{\alpha_0}{2} } c_0^{\frac{\alpha_1}{2} } > 0.
	\end{align*}
\end{proof}

\begin{lemma} \label{lemma:5.3}
	For any $c \in (0, c_0)$, let $(u_n) \subset B_{\rho_0}$ be such that $\normLp{u_n}{2}^2 \to c$ and $F_{\mu}(u_n) \to m(c)$. Then, there exist a $\beta_1 > 0$ and a sequence $(y_n) \subset \R^N$ such that
	\begin{align}
	\int_{B(y_n, R)} \abs{u_n}^2 dx \geq \beta_1 > 0, \qquad \text{for some } R > 0. \label{eqn:5.1}
	\end{align}
\end{lemma}
\begin{proof}
	We assume by contradiction that \eqref{eqn:5.1} does not hold. Since $(u_n) \subset B_{\rho_0}$ and $\normLp{u_n}{2}^2 \to c$, the sequence $(u_n)$ is bounded in $H$. From \cite[Lemma I.1]{LIONS1984-2} and since  $2 < q < 2^*$, we deduce that $\normLp{u_n}{q} \to 0,$ as $ n \to \infty.$
	At this point, \cref{lemma:5.6} implies that $F_{\mu} (u_n) \geq o_n(1)$. This contradict the fact that $m(c) <0$ and the lemma follows.
\end{proof}


\begin{proof} [Proof of \cref{theorem:LT-L}]
	We know from \cref{lemma:5.3}  and Rellich compactness theorem that there exists a sequence $(y_n) \subset \R^N$ such that 
	\begin{align*}
	u_n(x - y_n) \weakto u_c \neq 0 \quad\text{in } \Hoc.
	\end{align*}
	Our aim is to prove that $w_n(x) := u_n(x - y_n) - u_c(x) \to 0$ in $\Hoc$. Clearly
	\begin{align*}
	\normLp{u_n}{2}^2 &= \normLp{u_n(x-y_n)}{2}^2
	= \normLp{u_n(x-y_n) - u_c(x)}{2}^2 + \normLp{u_c}{2}^2 + o_n(1)\\
	&=\normLp{w_n}{2}^2 + \normLp{u_c}{2}^2 + o_n(1).
	\end{align*}
	Thus, we have
	\begin{align}
	\normLp{w_n}{2}^2 = \normLp{u_n}{2}^2 -  \normLp{u_c}{2}^2 + o_n(1) = c - \normLp{u_c}{2}^2 + o_n(1). \label{eqn:5.6}
	\end{align}
	By a similar argument, 
	\begin{align}
	\normLp{\diff w_n}{2}^2 = \normLp{\diff u_n}{2}^2 -  \normLp{\diff u_c}{2}^2 + o_n(1). \label{eqn:5.15}
	\end{align}
	More generally, taking into account that any term in $F_{\mu}$ fulfills the splitting properties of Brezis-Lieb \cite{BrezisLieb1983}, we have
	\begin{align*}
	F_{\mu}(w_n ) + F_{\mu}(u_c) =F_{\mu}(u_n(x-y_n)) + o_n(1),
	\end{align*}
	and, by the translational invariance, we obtain
	\begin{equation}\label{eqn:5.5}
	F_{\mu}(u_n) =F_{\mu}(u_n(x-y_n))
	=F_{\mu}(w_n) + F_{\mu}(u_c) + o_n(1).
	\end{equation}
	Now, we claim that
	\begin{align}\label{claimS}
	\normLp{w_n}{2}^2 \to 0. 
	\end{align}
	In order to prove this, let us denote $c_1:= \normLp{u_c}{2}^2 > 0$. By \eqref{eqn:5.6}, if we show that $c_1 = c$ then the claim follows. We assume by contradiction that $c_1 < c$.
	In view of \eqref{eqn:5.6} and \eqref{eqn:5.15}, for $n$ large enough, we have $\normLp{w_n}{2}^2 \leq c$ and $\normLp{\diff w_n}{2}^2 \leq \normLp{\diff u_n}{2}^2 < \rho_0$. Hence, we obtain that $w_n \in V(\normLp{w_n}{2}^2)$ and $F_{\mu}(w_n) \geq m\(\normLp{w_n}{2}^2\)$.
	Recording that $F_{\mu}(u_n) \to m(c)$, in view of \eqref{eqn:5.5}, we have
	\begin{align*}
	m(c) = F_{\mu}(w_n) + F_{\mu}(u_c) + o_n(1) \geq m\(\normLp{w_n}{2}^2\) + F_{\mu}(u_c) + o_n(1). 
	\end{align*}	
	Since the map $c \mapsto m(c)$ is continuous (see \cref{lemma:5.5}\eqref{point:5.5ii}) and in view of \eqref{eqn:5.6}, we deduce that
	\begin{align}
	m(c) \geq m(c-c_1) + F_{\mu}(u_c). \label{eqn:5.13}
	\end{align}
	We also have that $u_c \in V(c_1)$ by the weak limit. This implies that $F_{\mu}(u_c) \geq m(c_1)$. If $F_{\mu}(u_c) > m(c_1)$, then it follows from \eqref{eqn:5.13} and \cref{lemma:5.5}\eqref{point:5.5iii} that
	\begin{align*}
	m(c) > m(c-c_1) + m(c_1) \geq m(c -c_1 + c_1) = m(c),
	\end{align*}
	which is impossible. Hence, we have $F_{\mu}(u_c) = m(c_1)$, namely $u_c$ is a local minimizer on $V(c_1)$. So,  using \cref{lemma:5.5}\eqref{point:5.5iii} with the strict inequality, we deduce from \eqref{eqn:5.13} that
	\begin{align*}
	m(c) \geq m(c-c_1) + F_{\mu}(u_c) = m(c-c_1) + m(c_1) > m(c -c_1 + c_1) = m(c),
	\end{align*}
	which is impossible. Thus, the claim \eqref{claimS} follows and from \eqref{eqn:5.6} we deduce that
	$\normLp{u_c}{2}^2 = c$.\medskip
	
	Let us now show that $||\nabla w_n||_2^2 \to 0$. This will prove that $w_n \to 0$ in $\Hoc$ and completes the proof.
	In this aim first observe that 
	 in view of \eqref{eqn:5.15} and since $u_c \neq 0$,  we have $\normLp{\diff w_n}{2}^2 \leq \normLp{\diff u_n}{2}^2 < \rho_0$, for $n$ large enough.
	Hence $(w_n) \subset B_{\rho_0}$ and in particular it is bounded in $H$. Then by using the Gagliardo-Nirenberg inequality \eqref{Gagliardo-Nirenberg-I}, and by recalling $||w_n||_2^2 \to 0$ we also have $||w_n||_q^q \to 0$.
	Thus  
	\cref{lemma:5.6} implies that 
	\begin{equation}\label{firstpart}
	F_{\mu} (w_n) \geq \beta_0 \normLp{\diff w_n}{2}^2+o_n(1) \, \mbox{ where } \, \beta_0 > 0.
	\end{equation}
	  Now we remember that
	\begin{equation*}\label{LLLL}
	F_{\mu}(u_n) = F_{\mu}(u_c) + F_{\mu}(w_n) + o_n(1) \to m(c).
	\end{equation*} 
	Since $u_c \in V(c)$ by weak limit, we have that  $F_{\mu}(u_c) \geq m(c)$ and hence
	$F_{\mu}(w_n) \leq o_n(1)$. In view of \eqref{firstpart}, we then conclude that $\normLp{\diff w_n}{2}^2 \to 0$. 
\end{proof}


We end this section with,
\begin{proof}[Proof of \cref{thm-1}]
The fact that if $(u_n) \subset V(c)$ is such that $F_{\mu}(u_n) \to m(c)$ then, up to translation, $u_n \to u \in \mathcal{M}_c$ in $H$ follows from \cref{theorem:LT-L}.  In particular, it insures the existence of a minimizer for $F_{\mu}$ on $V(c)$. The fact that this minimizer is a ground state and that any ground state for $F_{\mu}$ on $S(c)$ belongs to $V(c)$ was proved in \cref{Lemma-structure}.
\end{proof}


\section{A condition for a uniform local existence result}\label{Section-4}
In this section we focus  on the local existence of solutions to the following Cauchy problem 
\begin{align}
\begin{cases} \label{NLS}
i \partial_t u + \Delta u + \mu u |u|^{q-2} + u |u|^{2^* - 2} = 0, \quad (t, x) \in \R \times \Rn, \quad N\geq 3 \\
u(0, x) = \varphi(x) \in \Hoc.
\end{cases}
\end{align}
Denoting $g: \C \to \C$ by $g(u) := \mu u |u|^{q-2} + u |u|^{2^* - 2}$,  \eqref{NLS} reads as
\begin{equation*}
i\partial_t u + \Delta u + g(u) = 0.
\end{equation*}

 Next we give the notion of integral equation associated with \eqref{NLS}.
In order to do that first we give another definition.
\begin{definition}
	If $N\geq 3$ the pair $(p, r)$ is said to be (Schr\"{o}dinger) admissible if
	\begin{align*}
	\dfrac{2}{p} + \dfrac{N}{r} = \dfrac{N}{2}, \qquad p, r \in [2, \infty].
	\end{align*}
\end{definition}
We shall work with two particular admissible pairs (see \cref{lemma:1T}):
\begin{align*}
(p_1, r_1) := \(\dfrac{4q}{(q - 2)(N-2)}, \dfrac{Nq}{q + N -2}\),
\end{align*}
and
\begin{align*}
(p_2, r_2) := \(\dfrac{4 \times 2^*}{(2^* - 2)(N-2)}, \dfrac{N \times 2^*}{2^* + N -2}\).
\end{align*}
Along with those couples we introduce the spaces $Y_{T}:= Y_{p_1, r_1, T} \cap Y_{p_2, r_2, T}$ and $X_{T}:= X_{p_1, r_1, T} \cap X_{p_2, r_2, T}$ equipped with the following norms:
\begin{align} \label{eqn:XT}
\norm{w}_{Y_T} = \norm{w}_{Y_{p_1, r_1, T}} + \norm{w}_{Y_{p_2, r_2, T}}, \quad\mbox{and}\quad \norm{w}_{X_T} = \norm{w}_{X_{p_1, r_1, T}} + \norm{w}_{X_{p_2, r_2, T}}. 
\end{align}
where for a generic function $w(t,x)$ defined on the time-space strip $[0, T)\times \Rn$ we have defined:
\begin{align*}
||w(t, x)||_{Y_{p,r,T}} = \(\int_{0}^{T} \normLp{w(t, \cdot)}{r}^p dt \)^{\frac{1}{p}} \hbox{ 
and }
||w(t, x)||_{X_{p,r,T}} = \(\int_{0}^{T} ||w(t, \cdot)||_{W^{1,r}(\Rn)}^p dt \)^{\frac{1}{p}}.
\end{align*}

\begin{definition} \label{def:1}
	Let $T >0$.
	We say that $u(t,x)$ is an integral solution of the Cauchy problem \eqref{NLS} on the time interval $[0, T]$ if:
	\begin{enumerate}
		\item $u \in \mathcal{C}([0, T], \Hoc) \cap X_T$;
		\item for all $t \in (0, T)$ it holds
		$
		u(t) = e^{it\Delta}\varphi - i\int_0^t e^{i(t-s)\Delta}g(u(s))ds.
		$
	\end{enumerate}
\end{definition}

The main result of this section is the following local existence result. We do not claim a real originality here, related versions already exist in the literature, see for example\cite[Theorem 2.5]{KenigMerle2006}. However, we believe convenient to the reader to provide a version specifically adapted to our problem and to give a proof of this result as self-contained as possible.

\begin{proposition}\label{prop:cauchy}
	There exists $\gamma_0  > 0$ such that if $\varphi \in \Hoc$ and  $T \in (0, 1]$ satisfy
	\begin{equation*} 
	\|e^{it\Delta}\varphi\|_{X_T} \leq \gamma_0, 
	\end{equation*}
	then there exists a unique integral solution $u(t,x)$ to \eqref{NLS} on the time interval $[0, T]$. Moreover
	$u(t,x)\in X_{p,r, T}$ for every admissible couple $(p,r)$ and satisfies the following conservation laws:
	\begin{equation}
	\label{eq:laws}
	F_\mu(u(t)) =  F_\mu(\varphi), \quad \normLp{u(t)}{2} = \normLp{\varphi}{2}, \quad \mbox{for all }t \in [0, T].
	\end{equation}
\end{proposition}

In order to prove \cref{prop:cauchy} we need some preliminary results. \medskip


Let us recall Strichartz's estimates that will be useful in the sequel (see for example \cite[Theorem 2.3.3 and Remark 2.3.8]{Cazenave2003semilinear} and \cite{KeelTao1998} for the endpoint estimates). 
\begin{proposition}\label{Strichartz}  Let $N\geq 3$ then 
	for every admissible pairs $(p, r)$ and $(\tilde p, \tilde r)$, there exists a constant $C> 0$ such that for every $T>0$, the following properties hold:
	\begin{itemize}
		\item[(i)] For every $\varphi \in \Lp{2}$, the function $t \mapsto e^{it\Delta}\varphi$ belongs to $Y_{p, r, T} \cap \mathcal{C}([0,T], \Lp{2})$ and
		\begin{align*}
		\norm*{e^{it\Delta} \varphi}_{Y_{p,r,T}} \leq C \norm{\varphi}_{2}.
		\end{align*}
		
		\item[(ii)] Let $F \in Y_{\tilde p ', \tilde r', T}$, where we use a prime to denote conjugate indices. Then the function
		\begin{align*}
		t \mapsto \Phi_F(t):= \int_{0}^{t}e^{i(t-s)\Delta}F(s) ds
		\end{align*}
		belongs to $Y_{p, r, T} \cap \mathcal{C}([0,T], \Lp{2})$ and 
		\begin{align*}
		\norm*{\Phi_F}_{Y_{p,r,T}} \leq C \norm{F}_{Y_{\tilde p ', \tilde r', T}}.
		\end{align*}
		
		\item[(iii)] For every $\varphi \in \Hoc$, the function $t \mapsto e^{it\Delta}\varphi$ belongs to $X_{p, r, T} \cap \mathcal{C}([0,T], \Hoc)$ and
		\begin{align*}
		\norm*{e^{it\Delta} \varphi}_{X_{p,r,T}} \leq C \norm{\varphi}_{\Hoc}.
		\end{align*}
		
	\end{itemize}
\end{proposition}
The following result will be useful in the sequel.
\begin{lemma} \label{lemma:1T}
	Let $N\geq 3$ and  $2 < \alpha \leq 2^*$ be given. Then the couple $(p,r)$ defined as follows
	\begin{align*}
	p := \dfrac{4\alpha}{(\alpha - 2)(N-2)} \qquad\mbox{and}\qquad r:= \dfrac{N\alpha}{\alpha + N -2}
	\end{align*}
	 is admissible. 
	Moreover for every admissible couple $(\tilde{p}, \tilde{r})$ there exists a constant $C > 0$ 
	such that for every $T >0$ the following inequalities hold:
	\begin{align}
	\norm*{\int_{0}^{t}e^{i(t-s)\Delta} [\nabla g_\alpha(u(s)) ]ds}_{Y_{\tilde{p},\tilde{r},T}}  &\leq C T^{\mu} \norm{\nabla u}_{Y_{p,r,T}}^{\alpha-1}, \label{eqn:2.17} \\
	\norm*{\int_{0}^{t}e^{i(t-s)\Delta} [g_\alpha(u(s)) - g_\alpha(v(s))] ds}_{Y_{\tilde{p},\tilde{r},T}}  &\leq C T^{\mu} (\norm{\nabla u}_{Y_{p,r,T}}^{\alpha-2} + \norm{\nabla v}_{Y_{p,r,T}}^{\alpha-2})  \norm{u-v}_{Y_{p,r,T}}, \label{eqn:2.18} 
	\end{align}
	where $g_\alpha(u) := u|u|^{\alpha - 2}$ and
	$
	\mu := \dfrac{(N-2)(2^*-\alpha)}{4} \geq 0.
	$
\end{lemma}
\begin{proof}	
	By direct calculations, one can check that
	\begin{align*}
	\dfrac{2}{p} + \dfrac{N}{r} = \dfrac{N}{2} \quad\mbox{and}\quad p, r \geq 2.
	\end{align*}
	Hence, $(p, r)$ is an admissible pair. Also it is easy to check that there exists a $C > 0$ such that :
	\begin{align}
	|g_\alpha'(u)| &\leq C|u|^{\alpha - 2}, \label{eqn:2.12}  \\
	|g_\alpha(u) - g_\alpha(v)| &\leq C|u - v|(|u|^{\alpha - 2} + |v|^{\alpha - 2}). \label{eqn:2.13} 
	\end{align}
	Combining \eqref{eqn:2.12} and the Chain Rule, gives
	\begin{align*}
	|\nabla g_\alpha(u) | = |g_\alpha'(u)\nabla u| \leq C |\nabla u| |u|^{\alpha - 2}. 
	\end{align*}
	Using H\"older's inequality,  we obtain that
	\begin{align*}
	\normLp{\nabla g_\alpha(u)}{r'} \leq C \normLp{|\nabla u| |u|^{\alpha - 2}}{r'} 
	\leq C \normLp{|\nabla u|}{r} \normLp{u}{r^*}^{\alpha - 2}
	\leq C \norm{\nabla u }^{\alpha - 1}_{r},
	\end{align*}
	where we also used the Sobolev embedding of $W^{1,r}(\R^N)$ into $L^{r^*}(\R^N)$ with $r^* := \dfrac{Nr}{N-r}$, see \cite[Theorem IX.9]{Brezis1983}.
	Hence, using  H\"{o}lder's inequality,
	\begin{align*}
	\norm{\nabla g_\alpha(u)}_{Y_{p', r', T}} &= \( \int_{0}^{T}  \normLp{\nabla g_\alpha(u)}{r'}^{p'} dt \)^{\frac{1}{p'}} 
	\leq C\( \int_{0}^{T}  || \nabla u ||^{(\alpha - 1)p'}_{r} dt \)^{\frac{1}{p'}} \\
	&\leq C T^{(\alpha-1)\(\frac{1}{(\alpha - 1) p'} - \frac{1}{p}\) } \( \int_{0}^{T}  || \nabla u ||^{p}_{r} dt \)^{\frac{\alpha - 1}{p}}
	= C T^{\mu} \norm{\nabla u}_{Y_{p,r,T}}^{\alpha-1}.
	\end{align*}
	At this point \eqref{eqn:2.17} follows by applying \cref{Strichartz} (ii). To establish \eqref{eqn:2.18} note that by   \eqref{eqn:2.13} and the H\"older's inequality, we have
	\begin{align*}
	\normLp{g_\alpha(u) - g_\alpha(v)}{r'} \leq C \normLp{|u - v|(|u|^{\alpha - 2} + |v|^{\alpha - 2})}{r'} 
	\leq C \normLp{u-v}{r} \normLp{|u| + |v|}{r^*}^{\alpha -2}.
	\end{align*}
	Hence, we can deduce that
	\begin{align*}
	\norm{g_\alpha(u) - g_\alpha(v)}_{Y_{p', r', T}} &= \( \int_{0}^{T}  \normLp{g_\alpha(u) - g_\alpha(v)}{r'}^{p'} dt \)^{\frac{1}{p'}} 
	\leq C \( \int_{0}^{T}  \normLp{u-v}{r}^{p'} \normLp{|u| + |v|}{r^*}^{(\alpha -2)p'} dt \)^{\frac{1}{p'}}\\
	&\leq C \( \int_{0}^{T}  \normLp{u-v}{r}^{p} dt \)^{\frac{1}{p}} \( \int_{0}^{T} \normLp{|u| + |v|}{r^*}^{\frac{(\alpha -2)pp'}{p-p'}} dt \)^{\frac{p-p'}{pp'}}
	\leq C T^{\mu} \norm{u-v}_{Y_{p,r,T}} \( \int_{0}^{T} \normLp{|u| + |v|}{r^*}^p dt \)^{\frac{\alpha-2}{p}}\\
	&= C T^{\mu} \norm{|u| + |v|}_{Y_{p,r^*,T}}^{\alpha-2}  \norm{u-v}_{Y_{p,r,T}} 
	\leq C T^{\mu} \(\norm{u}_{Y_{p,r^*,T}} + \norm{v}_{Y_{p,r^*,T}} \)^{\alpha-2}  \norm{u-v}_{Y_{p,r,T}}.
	\end{align*}
	The inequality \eqref{eqn:2.18} follows by applying the previous Sobolev embedding and \cref{Strichartz} (ii).
\end{proof}





In order to prove \cref{prop:cauchy} we shall need two lemmas from Functional Analysis.
\begin{lemma}
	\label{lem:reflexive}
	For all $1 < p, r < \infty$, $X_{p, r, T}$ is a separable reflexive Banach space.
\end{lemma}
\begin{proof}
	This is a direct consequence of Phillips' theorem, see \cite[Chapter IV, Corollary 2]{Diestel-Uhl}.
\end{proof}
\begin{lemma}\label{complete}
	For all $R, T > 0$ the metric space $(B_{R,T}, d)$, where
	\begin{equation*}
	B_{R,T} := \{u \in X_T: \|u\|_{X_T} \leq R\},
	\end{equation*}
	and
	\begin{equation*}
	d(u, v) := \|u - v\|_{Y_T}
	\end{equation*}
	is complete.
\end{lemma}
\begin{proof}
	Let $(u_n)$ be a Cauchy sequence. Since $Y_T$ is a Banach space, there exists $u \in Y_T$ such that
	\begin{equation*}
	\lim_{n\to\infty}\|u_n - u\|_{Y_T} = 0.
	\end{equation*}
	It remains to show that $u \in B_{R,T}$.
	
	By taking a subsequence, we can assume that $l_1 := \lim_{n\to\infty}\|u_n\|_{X_{p_1, r_1, T}}$
	and $l_2 := \lim_{n\to\infty}\|u_n\|_{X_{p_2, r_2, T}}$ exist.
	By \cref{lem:reflexive}, there exists a subsequence of $(u_n)$ which converges weakly in $X_{p_1, r_1, T}$.
	In particular, this sequence converges in the sense of distributions and hence
	the limit equals $u$. Thus,
	\begin{equation*}
	\|u\|_{X_{p_1, r_1, T}} \leq l_1.
	\end{equation*}
	Similarly,
	\begin{equation*}
	\|u\|_{X_{p_2, r_2, T}} \leq l_2.
	\end{equation*}
	Taking the sum, we get $\|u\|_{X_T} \leq l_1 + l_2 \leq R$.	
\end{proof}

\begin{proof}[Proof of \cref{prop:cauchy}]
	\textbf{Step 1. Existence and uniqueness in $B_{2\gamma_0,T}$ for $\gamma_0$ small enough.}
	For any $u \in X_T$ and $t \in [0, T]$, we define
	\begin{equation}
	\label{eq:Phi-def}
	\Phi(u)(t) := e^{it\Delta}\varphi + i\int_0^te^{i(t-s)\Delta}g(u(s))ds.
	\end{equation}
	We claim that, if $\gamma_0 >0$ is small enough, then $\Phi$ defines a contraction on 
	the metric space $(B_{2\gamma_0,T}, d)$ (see  \cref{complete}).
	
	Let $u \in B_{2\gamma_0, T}$ and consider any admissible pair $(\tilde p, \tilde r)$. Let $T \in (0,1]$ and apply \cref{lemma:1T}. We deduce from \eqref{eqn:2.17} and \eqref{eq:Phi-def} that
	\begin{align*}
	\norm{\nabla \Phi(u)-e^{it\laplace} \nabla \varphi}_{Y_{\tilde p, \tilde r,T}} \leq C \norm{\nabla u}_{Y_{p_1,r_1,T}}^{q-1} + C \norm{\nabla u}_{Y_{p_2,r_2,T}}^{2^*-1} \leq C 2^q \gamma_0^{q - 1}, \quad \forall u\in 
	B_{2\gamma_0, T} . 
	\end{align*}
	Similarly, we deduce from \eqref{eqn:2.18} (applied with $v = 0$) that
	\begin{align*}
	\begin{split}
	\norm{\Phi(u)- e^{it\laplace} \varphi}_{Y_{\tilde p, \tilde r,T}} \leq C \norm{\nabla u}_{Y_{p_1,r_1,T}}^{q-2} \norm{u}_{Y_{p_1,r_1,T}} + C \norm{\nabla u}_{Y_{p_2,r_2,T}}^{2^*-2} \norm{u}_{Y_{p_2,r_2,T}}
	\leq C 2^q \gamma_0^{q-1}, \quad \forall u\in 
	B_{2\gamma_0, T} . 
	\end{split}
	\end{align*}
	In particular if we choose $(\tilde p, \tilde r)=(p_1, r_1)$ and $(\tilde p, \tilde r)=(p_2, r_2)$ then
	$$\|\Phi(u)\|_{X_T}\leq \gamma_0 + C 2^q \gamma_0^{q-1}$$
	and hence if $\gamma_0 >0$ is small enough in such a way that 
	$C 2^{q+2} \gamma_0^{q-1}\leq \gamma_0$, then $B_{2\gamma_0,T}$ is an invariant set of $\Phi$.
	
	Now, let $u, v \in B_{2\gamma_0, T}$. By \eqref{eqn:2.18}, we have for every admissible pair $(\tilde p, \tilde r)$
	\begin{align*}
	\|\Phi(u) - \Phi(v)\|_{Y_{\tilde p, \tilde r, T}}  &\leq C\big(\norm{\nabla u}_{Y_{p_1,r_1,T}}^{q-2} + \norm{\nabla v}_{Y_{p_1,r_1,T}}^{q-2}\big)\norm{u-v}_{Y_{p_1, r_1, T}} 
	+ C\big(\norm{\nabla u}_{Y_{p_2,r_2,T}}^{2^*-2}+ \norm{\nabla v}_{Y_{p_2,r_2,T}}^{2^*-2}\big)\norm{u-v}_{Y_{p_2, r_2, T}} \\
	&\leq C 2^q \gamma_0^{q - 2}(\norm{u-v}_{Y_{p_1, r_1, T}} + \norm{u-v}_{Y_{p_2, r_2, T}}), \quad \forall u,v \in 
	B_{2\gamma_0, T}.
	\end{align*}
	In particular if we choose $(\tilde p, \tilde r)=(p_1, r_1)$ and $(\tilde p, \tilde r)=(p_2, r_2)$ then
	$$\|\Phi(u) - \Phi(v)\|_{Y_{T}} \leq C 2^{q+1} \gamma_0^{q - 2}\|u - v\|_{Y_{T}}$$
	and if we choose $\gamma_0 >0$ small enough in such a way that $C 2^{q+1} \gamma_0^{q - 2}<\frac 12$
	then $\Phi$ is a contraction on $(B_{2\gamma_0, T}, d)$. In particular $\Phi$ has one unique
	fixed point in this space.
	The property  $u \in C([0, T], \Hoc)$ and $u\in X_{p,r,T}$ for every admissible couple $(p,r)$ is straightforward and     follows 
         by Strichartz estimates. 
         
	\textbf{Step 2. Uniqueness in $X_T$.} Assume $u_1(t,x)$ and $u_2(t,x)$ are two fixed points
of $\Phi$ in the space $X_T$. 
	We define 
	$T_0=\sup \{\bar T \in [0,T]| \sup_i \|u_i(t,x)\|_{X_{\bar T}} \leq 2\gamma_0\}$.
	It is easy to show that $T_0\in (0, \bar T]$ and arguing as in step 1 the operator $\Phi$ is a contraction on $(B_{2\gamma_0, T_0}, d)$. Hence by uniqueness of the fixed point in this space necessarily 
	$u_1(t,x)=u_2(t,x)$ in $X_{T_0}$.  Moreover since $u_i(t,x)\in \mathcal C([0, T_0]; H)$
	we have $u_1(T_0,x)=u_2(T_0,x)=\psi(x)$. Hence at time $T_0$ the solutions coincide and starting from $T_0$
	(that we can also identify with $T_0=0$ by using the traslation invariance w.r.t. to time of the equation),
	we can apply again the step 1 in the ball $(B_{2\gamma_0, \tilde T}, d)$ with initial condition
	$\psi(x)$, where 
	$\tilde T>0$ is such that $\|e^{it\Delta}\psi\|_{X_{\tilde T}} \leq \gamma_0$. Again by uniqueness
	of the fixed point of $\Phi$ in the space $(B_{2\gamma_0, \tilde T}, d)$ we deduce that
	$u_1(t,x)=u_2(t,x)$ in $X_{T_0+\tilde T}$, hence contradicting the definition of $T_0$ unless $T_0=T$. 
	
	
	\textbf{Step 3. Conservation laws.} The proof of \eqref{eq:laws} is 
	rather classical. In particular 
	it follows by Proposition 1 and Proposition 2 in \cite{Ozawa06}. Another possibility is to follow the proof of
	Propositions 5.3 and 5.4 in \cite{Ginibre}, that
	 can be repeated {\em mutatis mutandis} 
	in the context of \eqref{NLS}. The minor modification compared with  \cite{Ginibre} is that we use the end-point Strichartz   
	estimate in order to treat the Sobolev critical nonlinearity.    
\end{proof}


\section{Orbital stability}\label{Section-5}

We shall prove in this section that the set $\mathcal{M}_c$ defined in  \eqref{set-stable} is orbitally stable. 
In particular a nontrivial point concerns the fact that the local solutions, whose existence has been established in \cref{Section-4}, 
can be extended to global solutions provided that the initial datum is close to $\mathcal{M}_c$.
The main difficulty is related to the criticality of the nonlinearity in \eqref{NLS}, which implies that
an a priori bound on the Mass and the Energy is not sufficient
to exclude a finite-time blow-up.
We will overcome this issue by deducing
from the uniform local well-posedness (Proposition~\ref{prop:cauchy})
a uniform lower bound on the time of existence of the solution
corresponding to initial data close to a set which is compact up to translations, see Theorem~\ref{theorem:2}.

\medskip


To simplify the next statement we denote by $u_{\varphi}(t)$ the integral solution associated with \eqref{NLS} and we denote by $T_{\varphi}^{max}$ its maximal time of existence.

\begin{theorem} \label{theorem:1}
	Let $v \in \mathcal{M}_c $. Then, for  every $\varepsilon > 0$ there exists $\delta > 0$ such that:
	\begin{align} \label{eqn:2.9}
	\forall \varphi \in \Hoc \mbox{ s.t. } ||\varphi - v||_{\Hoc} < \delta \Longrightarrow \sup_{t \in [0, T_{\varphi}^{max})}  \dist{u_{\varphi}(t)}{\mathcal{M}_c} < \varepsilon. 
	\end{align}
	In particular we have
	\begin{align} \label{eqn:2.10}
	u_{\varphi}(t) = m_{c}(t) + r(t), \quad \forall t \in [0, T_{\varphi}^{max}), \mbox{ where } m_{c}(t) \in \mathcal{M}_c, \, \norm{r(t)}_{\Hoc} < \varepsilon.
	\end{align}
\end{theorem}
\begin{proof}
	Suppose the theorem is false. Then there exists $(\delta_n) \subset \R^+$  a decreasing sequence converging to $0$ and $(\varphi_n)  \subset \Hoc$ satisfying
	\begin{equation*}
	||\varphi_n - v||_{\Hoc} < \delta_n
	\end{equation*}
	and
	\begin{equation*}
	\sup_{t \in [0, T_{\varphi_n}^{max})} \dist{u_{\varphi_n}(t)}{\mathcal{M}_c} > \varepsilon_0,
	\end{equation*}
	for some $\varepsilon_0 >0$.  We observe that $||\varphi_n||_2^2 \to c$ and, by continuity of $F_{\mu}$, $F_{\mu}(\varphi_n) \to m(c)$. By conservation laws, for $n \in \N$ large enough, $u_{\varphi_n}$ will remains inside of $B_{\rho_0}$ for all $t \in [0, T_{\varphi_n}^{max})$.
	Indeed, if for some time $\overline{t}>0$ $||\nabla u_{\varphi_n}(\overline{t})||_2^2 = \rho_0$ then, in view of \cref{Lemma-structure} (\ref{point:5L.5i}) we have that $F_{\mu}(u_{\varphi_n}(\overline{t})) \geq 0$ in contradiction with $m(c) <0$. Now let $t_n>0$ be the first time such that
	$\dist{u_{\varphi_n}(t_n)}{\mathcal{M}_{c}} = \varepsilon_0$ and set $u_n := u_{\varphi_n}(t_n)$.
	By conservation laws, 
	$(u_n) \subset B_{\rho_0}$ satisfies $\normLp{u_n}{2}^2 \to c$ and $F_{\mu}(u_n) \to m(c)$ and thus, in view of \cref{theorem:LT-L}, it converges, up to translation, to an element of $\mathcal{M}_c$. Since $\mathcal{M}_c$ is invariant under translation this contradicts the equality 
	$\dist{u_n}{\mathcal{M}_{c}} = \varepsilon_0 >0$.
\end{proof}
The rest of this section is devoted to showing that $T_{\varphi}^{max}=\infty$ and it will conclude the proof of
\cref{thm-2}.

\begin{proposition} \label{proposition:1}
	Let $\mathcal{K} \subset \Hoc \setminus \{0\}$ be compact up to translation and assume that $(p, r)$ is an admissible pair with $p \neq \infty$. Then, for every $\gamma > 0$ there exists $\varepsilon = \varepsilon(\gamma) > 0$ and $T = T(\gamma) > 0$ such that 
	\begin{align*}
	\sup_{\{\varphi \in \Hoc| \dist{\varphi}{\mathcal{K}} < \varepsilon\}} \norm*{e^{it\Delta} \varphi}_{X_{p,r,T}} < \gamma.
	\end{align*}
\end{proposition}
\begin{proof}
	We first claim, for every $\gamma > 0$, the existence of a $T > 0$ such that
	\begin{align}\label{ajoutL5}
	\sup_{\varphi \in \mathcal{K}} \norm*{e^{it\Delta} \varphi}_{X_{p,r,T}} < \dfrac{\gamma}{2}.
	\end{align}
	If it is not true then there exists sequences $(\varphi_n) \subset \mathcal{K}$ and $(T_n) \subset \R^+$
	such that $T_n \to 0$ and
	\begin{align}\label{eqn:1t}
	\norm*{e^{it\Delta} \varphi_n}_{X_{p,r,T_n}} \geq \overline{\gamma} 
	\end{align}
	for a suitable $\overline{\gamma}>0$. Since $\mathcal{K}$ is compact up to translation, passing to a subsequence, there exists a sequence $(x_n) \subset \R^N$ such that
	$$ \tilde{\varphi_n}(\cdot) := \varphi_n(\cdot - x_n) \overset{ H }\to \varphi(\cdot)$$
	for a $\varphi \in \Hoc$. By continuity (induced by Strichartz's estimates) we have, for every $\bar T>0$,
	\begin{align}\label{ajoutL2}
	\norm*{e^{it\Delta} \tilde{\varphi}_n}_{X_{p,r,\bar T}} \to \norm*{e^{it\Delta} \varphi}_{X_{p,r,\bar T}}.
	\end{align}	
	Also, recording the translation invariance of Strichartz's estimates we get from \eqref{eqn:1t} that
	\begin{equation}\label{ajoutL1}
	\norm*{e^{it\Delta} \tilde{\varphi}_n}_{X_{p,r,T_n}}  = \norm*{e^{it\Delta} \varphi_n}_{X_{p,r,T_n}} \geq \overline{\gamma}. 
	\end{equation}
	Now, by \cref{Strichartz} (iii), we have $e^{it\Delta} \varphi \in X_{p,r, 1}$, namely the function 
	$$[0,1]\ni t\to  g(t):=||e^{it\Delta} \varphi||_{W^{1,r}(\Rn)}^p$$
	belongs to $L^1([0,1])$. Then by 
	the Dominated Convergence Theorem we get $\|\chi_{[0, \tilde T]}(t) g(t)\|_{L^1([0,1])} \to 0$
	as $\tilde T\to 0$,
	namely $\norm*{e^{it\Delta} \varphi}_{X_{p,r,\tilde T}}^p 
	\to 0$ as $\tilde T \to 0$. Hence, we can choose $\bar T >0$ such that
	\begin{align}\label{ajoutL3}
	\norm*{e^{it\Delta} \varphi}_{X_{p,r,\bar T}} < \overline{\gamma}.
	\end{align}
	At this point gathering \eqref{ajoutL2}- \eqref{ajoutL3} we get a contradiction and the claim holds.
	Now, fix a $T>0$ such that \eqref{ajoutL5} holds.  By \cref{Strichartz} (iii), we have
	\begin{align*}
	\norm*{e^{it\Delta} \eta}_{X_{p,r,T}} \leq C ||\eta||_{\Hoc}, \qquad \forall \eta \in \Hoc.
	\end{align*}
	Thus, assuming  that $\displaystyle ||\eta||_{\Hoc} < \frac{\gamma}{2C} := \varepsilon$, we obtain that
	\begin{align*}
	\norm*{e^{it\Delta} \eta}_{X_{p,r,T}} < \frac{\gamma}{2}.
	\end{align*}
	Summarizing, we get that, for all $\varphi \in \mathcal{K}$ and all $\eta \in \Hoc$ such that $||\eta||_{\Hoc} < \varepsilon$,
	\begin{align*}
	\norm*{e^{it\Delta} (\varphi + \eta)}_{X_{p,r,T}} \leq \norm*{e^{it\Delta} \varphi}_{X_{p,r,T}} + \norm*{e^{it\Delta} \eta}_{X_{p,r,T}} < \gamma.
	\end{align*}
	This implies the proposition.
\end{proof}

\begin{proposition} \label{proposition:2}
	Let $\mathcal{K} \subset \Hoc \setminus \{0\}$ be compact up to translation. Then, for every $\gamma > 0$ there exists $\varepsilon = \varepsilon(\gamma) > 0$ and $T = T(\gamma) > 0$ such that 
	\begin{align*}
	\sup_{\{\varphi \in \Hoc| \dist{\varphi}{\mathcal{K}} < \varepsilon\}} \norm*{e^{it\Delta} \varphi}_{X_{T}} < \gamma.
	\end{align*}
\end{proposition}
\begin{proof}
	We apply \cref{proposition:1} twice with the admissible pairs $(p_1, r_1)$ and $(p_2, r_2)$. Then, the proposition follows from the definition of the norm $X_T$ given in \eqref{eqn:XT}.
\end{proof}

\begin{theorem} \label{theorem:2}
	Let  $\mathcal{K} \subset \Hoc \setminus \{0\}$ be compact up to translation. Then there exist $\varepsilon_0 > 0$ and $T_0>0$ such that the Cauchy problem \eqref{NLS}, where $\varphi$ satisfies $\dist{\varphi}{\mathcal{K}} < \varepsilon_0$, has a unique solution on the time interval $[0, T_0]$ in the sense of \cref{def:1}.
\end{theorem}
\begin{proof}
	We apply \cref{proposition:2} where $\gamma = \gamma_0$ is given in \cref{prop:cauchy}. Then \cref{prop:cauchy} guarantees that the theorem holds for $\varepsilon_0 = \varepsilon (\gamma_0) >0$ and $T_0 = \min\{ T(\gamma_0),1 \} >0$. 
\end{proof}

\begin{theorem} \label{theorem:3}
	Let $\mathcal{M}_c$ be defined in \eqref{set-stable}. Then there exists a $\delta_0 > 0$ such that, if $\varphi \in \Hoc$ satisfies $\dist{\varphi}{\mathcal{M}_c} < \delta_0$ the corresponding solution to \eqref{NLS} satisfies $T_{\varphi}^{max} = \infty$.
\end{theorem}
\begin{proof}
	We make use of \cref{theorem:2} where we choose $\mathcal{K} = \mathcal{M}_c$. By \cref{theorem:1}, we can choose a $\delta_0 >0$ such that \eqref{eqn:2.9} and \eqref{eqn:2.10} holds for $\varepsilon = \varepsilon_0$ where $\varepsilon_0 >0$ is given in \cref{theorem:2}. Then \cref{theorem:1} guarantees that the solution 
	$u_{\varphi}(t)$ where $\dist{\varphi}{\mathcal{M}_{c}} < \delta_0$ satisfies $\dist{u_{\varphi}(t)}{\mathcal{M}_{c}} < \varepsilon_0$ up to the maximum time of existence $T_{\varphi}^{max} \geq T_0$. Since, at any time in $(0, T_{\varphi}^{max})$ we can apply again \cref{theorem:2} that guarantees an uniform additional time of existence $T_0 >0$, this contradicts the definition of $T_{\varphi}^{max}$ if $T_{\varphi}^{max} < \infty$.
\end{proof}

At this point we can give,
\begin{proof}[Proof of \cref{thm-2}]
	The fact that $\mathcal{M}_c$ is compact, up to translation, was established in \cref{theorem:LT-L}. The orbital stability of $\mathcal{M}_c$, in the sense of 
	\cref{def:stability} follows from \cref{theorem:1} and \cref{theorem:3}. 	
\end{proof}


\renewcommand{\bibname}{References}
\bibliographystyle{plain}
\bibliography{References}
\vspace{0.25cm}

\end{document}